\documentclass[12pt, english]{amsart}
\usepackage[utf8]{inputenc}
\usepackage{xcolor}
\usepackage{array}
\usepackage[margin=1in]{geometry}

\newcommand{\rev}{}
\usepackage{float}
\usepackage{amsmath,amssymb,amsthm}
\usepackage{stmaryrd}
\usepackage{graphicx}
\usepackage{caption}
\usepackage{hyperref}
\usepackage{url}
\usepackage{mathrsfs}
\usepackage[all]{xy}
\usepackage{multicol}
\usepackage{mathtools}

\usepackage{color}
\usepackage[symbol]{footmisc}

\theoremstyle{plain}
\newtheorem{theorem}{Theorem}[section]

\newtheorem{corollary}[theorem]{Corollary}
\newtheorem{lemma}[theorem]{Lemma}

\newtheorem{proposition}[theorem]{Proposition}

\newtheorem*{thm*}{Theorem}

\theoremstyle{definition}

\newtheorem{remark}[theorem]{Remark}

\newtheorem{definition}[theorem]{Definition}

\newtheorem{example}[theorem]{Example}

\newcommand{\R}{\mathbb{R}}
\newcommand{\Z}{\mathbb{Z}}

\newcommand{\Q}{\mathbb{Q}}

\newcommand{\C}{\mathbb{C}}

\DeclareMathOperator{\Spec}{Spec}
\DeclareMathOperator{\Frac}{Frac}

\DeclareMathOperator*{\argmin}{arg\,min}

\title{The Reflection-Invariant Bispectrum: Signal Recovery in the Dihedral Model} 
\author{Dan Edidin, Josh Katz}
\address{Department of Mathematics, University of Missouri, Columbia, MO 65211}
\email{edidind@missouri.edu, jkgbh@mail.missouri.edu}
\thanks{The authors were supported by BSF grant 2020159 and NSF grant DMS2205626.}
\date{\today}
\begin{document}

\begin{abstract}
 We study the problem of signal recovery in the dihedral
 multi-reference alignment (MRA) model, where a signal is observed
 under random actions of the dihedral group and corrupted by additive
 noise. While previous work has shown that cyclic invariants of degree
 three (the bispectrum) suffice to recover generic signals up to
 circular shift \cite{bendory2017bispectrum}, the dihedral setting
 introduces new challenges due to the group’s non-abelian
 structure. In particular, reflections prevent the diagonalization of
 the third moment tensor in the Fourier basis, making classical
 bispectrum techniques inapplicable.

In this work, we prove that the orbit of a generic signal in the
$n$-dimensional standard representation of the $2n$-element dihedral
group $D_{n}$ is uniquely determined by its invariant tensors of
degree at most three. This resolves an open question posed
in~\cite{bendory2022dihedral}, and establishes that the sample
complexity for dihedral MRA with uniform distribution is
$\omega(\sigma^6)$, matching the cyclic case.  \rev{Along the way we
  prove a result of independent interest
  (Theorem~\ref{thm.orthogonal}), namely that invariants of degree at
  most three separate generic real orbits in band-limited
  representations of the orthogonal group $O(2)$.

While frequency marching becomes computationally impractical in the
dihedral setting, we show numerically that a simple optimization
algorithm reliably recovers the signal from third-order moments, even
with random initialization.  Our results establish the dihedral model
as both a challenging and viable alternative to the cyclic model—one
that better reflects practical symmetry constraints and deserves
further exploration in both theoretical and applied settings.}

\end{abstract}
\maketitle

\section{Introduction}
We study the dihedral multi-reference alignment (MRA) model
\begin{equation} \label{eq:mra}
	y = g\cdot x + \varepsilon, \quad g\sim\rho, \quad \varepsilon\sim N(0,\sigma^2I),
\end{equation}
where 
\begin{itemize}
	\item   $x\in\C^n$ is a fixed (deterministic) signal to be estimated; 
	\item $g$ is a random element of the dihedral group $D_{n}$ 
    which acts on the signal by circular translation and reflection (see Figure~\ref{fig:example});   
\end{itemize}
Our goal is to recover $x$ from $t$ observations of the form 
\begin{equation} \label{eq:mra_observations}
	y_i =g_i\cdot x + \varepsilon_i, \quad i=1,\ldots,t,
\end{equation}
while the corresponding group elements $g_1,\ldots,g_t$ are unknown. 
\begin{figure}[H]
    \centering
    \begin{minipage}{0.3\columnwidth}
        \centering
        \includegraphics[width=\columnwidth]{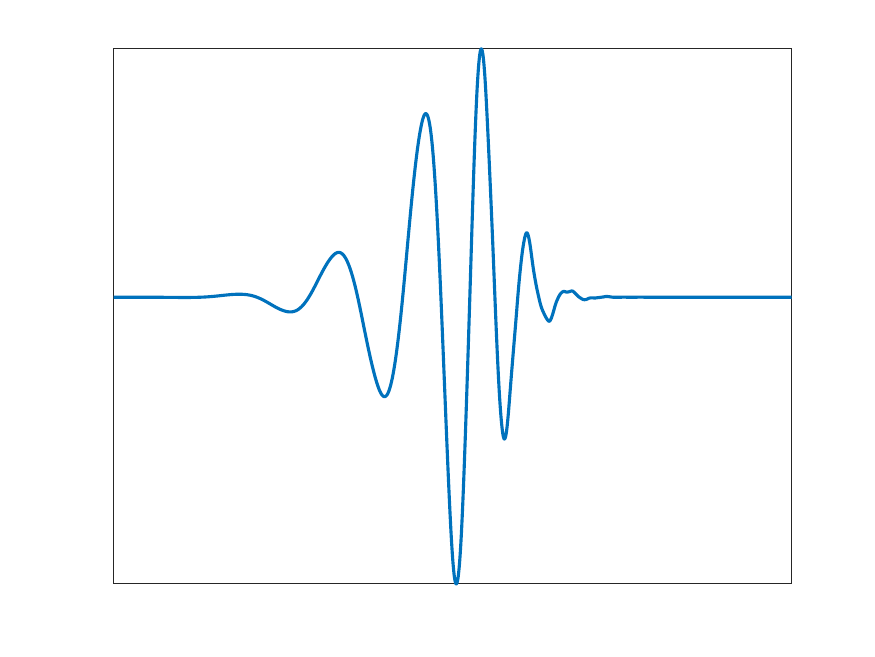}
        
    \end{minipage}
    \hfill
    \begin{minipage}{0.3\columnwidth}
        \centering
        \includegraphics[width=\columnwidth]{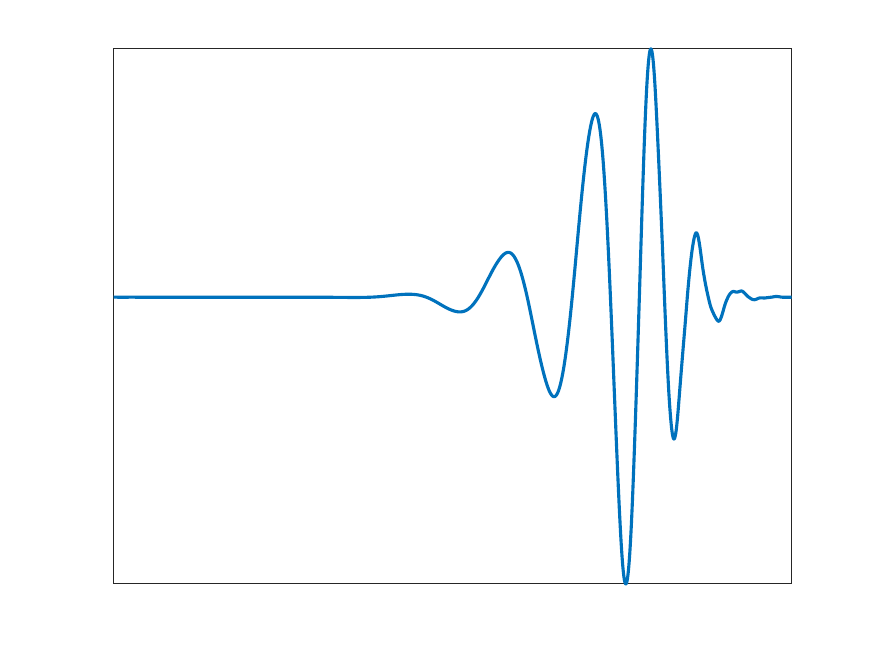}
        
    \end{minipage}
    \hfill
    \begin{minipage}{0.3\columnwidth}
        \centering
        \includegraphics[width=\columnwidth]{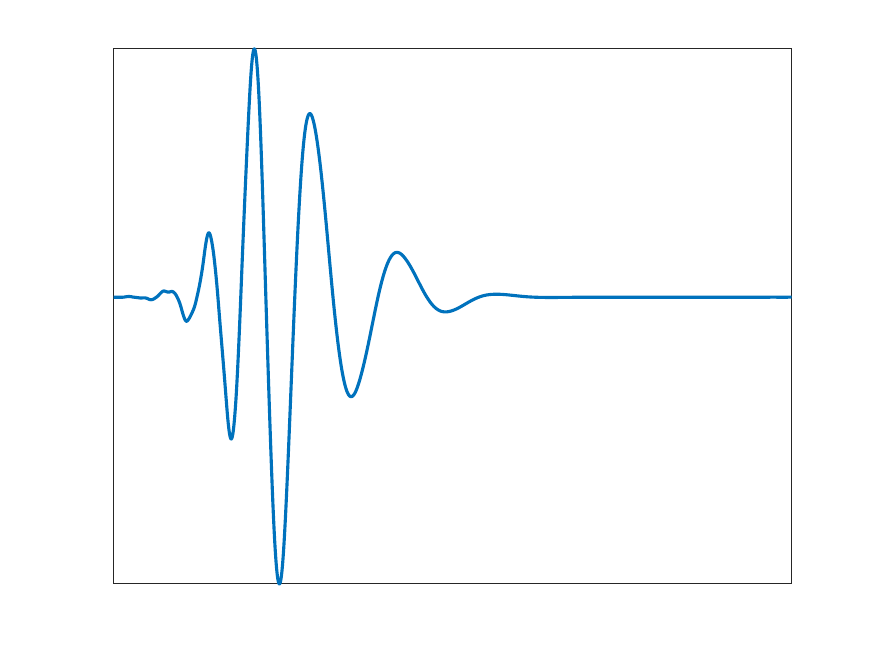}
        
    \end{minipage}
    
	\caption{\label{fig:example} An example of the action of the dihedral group. From left to right: the original signal, the signal after a circular shift, and the signal after a shift followed by reflection. The MRA problem~\eqref{eq:mra} entails estimating a signal, up to a global circular shift and reflection, from multiple noisy {copies of the signal}  acted upon by random elements of the dihedral group.}
\end{figure}

Because of its connection to cryo-EM, the MRA problem has been extensively
studied for various groups and representations~\cite{bandeira2014multireference,bendory2017bispectrum,bandeira2023estimation,abbe2017sample,abbe2018multireference,boumal2018heterogeneous,abbe2018estimation,perry2019sample,ma2019heterogeneous,bandeira2020optimal,bandeira2020non,romanov2021multi,abas2022generalized,hirn2019wavelet,aizenbud2021rank,katsevich2020likelihood,fan2020likelihood,ghosh2021multi,gao2019iterative,brunel2019learning}.
When the signal-to-noise ratio is extremely low, as is the case for cryo-EM measurements, there is no way to estimate the unknown group elements, but it can be shown that the moment tensors of the unknown signal can be accurately approximated with enough samples~\cite{abbe2018estimation}. This reduces the MRA problem
to the problem of recovering an orbit from its moment tensors. However, the sample complexity (the minimal number of measurements necessary for accurate approximation) grows
exponentially in the degrees of the moments, so to efficiently solve the MRA problem it is necessary to recover almost all signals from invariants of the lowest possible degree. 
A number of results on the sample complexity of MRA problems
involving cyclic and rotation groups $SO(2), SO(3)$ have previously
appeared in the literature~\cite{bandeira2020non,bandeira2023estimation, ma2019heterogeneous, janco2022accelerated,edidin2024orbit}. 

The problem of orbit recovery from invariant polynomials of low degree has also been studied in the context of equivariant machine learning. In addition to the rotation groups, researchers in this domain consider the permutation action of the symmetric group $S_n$ on $n \times d$ matrices, since this representation naturally arises in deep learning~\cite{balan2022permutation,segol2020universal, tabaghi2023universal, dym2022low,cohen2016group,kondor2018clebschgordan,blumsmith2023machine}. In~\cite{edidin2025orbit} 
the authors of this work gave numerical evidence that the third order invariant of the symmetric group action on the space of $n\times d$ matrices is sufficient to recover the orbit of a generic matrix when $d \gtrapprox \sqrt{n}$. 

In MRA models where the random group elements are drawn from a uniform distribution,
the coefficients of the moment tensors are invariant
functions on the representation. In other models,
the distribution of random elements is assumed to follow a `generic'
but unknown distribution. In this case the moments can be viewed as giving
invariants on $L^2(G) \times V$~\cite{bendory2022dihedral}. 
In particular,
the sample complexity of cyclic MRA for the uniform distribution is known
to be $\omega(\sigma^6)$~\cite{bandeira2020optimal}. In~\cite{abbe2018multireference} it was proved that for aperiodic distributions on the cyclic group $\Z_n$
the sample complexity of MRA is $\omega(\sigma^4)$. Likewise for the action
of the $2n$-element dihedral group $D_{n}$ on the subregular representation
$\C^n$ by cyclic rotations and reflections the sample
complexity is known to be $\omega(\sigma^4)$ if the probability distribution is assumed
to be generic~\cite{bendory2022dihedral}. In \cite{bendory2022dihedral} the authors ask
if the sample complexity for the uniform distribution is also $\omega(\sigma^6)$ for this representation of the dihedral group. Our work answers this question affirmatively.
Precisely, we have the following corollary of Theorem~\ref{thm.dihedral}.
\begin{corollary} \label{cor.dn-samplecomplexity}
Consider the multi-reference alignment model~\eqref{eq:mra} 
where the $g_i$ are drawn uniformly from  $D_n$ 
and $x \in \C^n$ is a generic vector. Then the minimal number of observations $N$ required for accurate recovery of $x$ is $N/\sigma^6\to\infty$.
\end{corollary}
\rev{\begin{remark}
For other representations of finite groups a number of results analogous to Corollary~\ref{cor.dn-samplecomplexity} are
proved in the literature. In particular in~\cite{bandeira2023estimation} it is proved that invariants of degree at most three separate generic orbits in the regular representation of a finite abelian group. The authors of this paper extended this to the regular representation of an arbitrary finite group in~\cite{edidin2025orbit} and in the same paper gave an example 
of a representation of $D_n$ which contains a copy of $\C^n$ but for which invariants of degree at most cannot separate
generic orbits.\end{remark}}
\rev{\begin{remark} For the case of the cyclic MRA problem the result of~\cite{abbe2018multireference} shows that
aperiodicity of the distribution is necessary and sufficient for $\omega(\sigma^4)$ sample complexity. By constrast, for dihedral MRA we do not have an explicit condition which differentiates between $\omega(\sigma^4)$ sample complexity seen
for a generic distribution and the $\omega(\sigma^6)$ sample complexity seen for the uniform distribution.\end{remark}}

In addition to the theoretical question of sample complexity, researchers have also explored the algorithmic challenges inherent in signal recovery. Significant progress has been made for the cyclic MRA model. It is well known that for the cyclic third moment, or its Fourier equivalent, the bispectrum determines the orbit of a generic signal~\cite{sadler1989shift}. In~\cite{bendory2017bispectrum}  an efficient frequency marching algorithm for reconstruction was given which takes advantage of the fact the bispectrum consists of monomials.

We establish that for the dihedral model, the third moment still determines the orbit of a generic signal, despite the fact that the frequency marching strategy becomes combinatorially infeasible. This discrepancy arises due to fundamental algebraic differences: the cyclic group's abelian structure allows for the diagonalization of the third moment tensor when computed in the Fourier basis (the bispectrum), transforming a complicated polynomial system into a tractable system of monomial equations. By contrast the dihedral group's non-abelian nature prevents such diagonalization, leaving no simplification to a purely monomial system. We will see that when we compute the third dihedral moment with respect to the Fourier basis we obtain a tensor which can be viewed as a reflection invariant analogue of the classical bispectrum.

To address these challenges, we investigate a direct optimization approach, attempting signal recovery by minimizing a polynomial loss function measuring the discrepancy between the true and computed dihedral third moments. In the cyclic model, this was shown to converge to the correct signal with random initialization in~\cite{bendory2017bispectrum}. We adapt the MATLAB code from~\cite{bendory2017bispectrum} and we show that as in the case of cyclic MRA, the dihedral optimization algorithm converges to the correct signal with random initialization. As the signal length increases, the error of the estimates is indistinguishable in the cyclic and dihedral models, although it remains higher for the dihedral case than for the cyclic case at small signal lengths. These optimization-based numerical experiments highlight the feasibility and robustness of the proposed approach, reinforcing its practical applicability. \rev{The code to reproduce all of our experiments is publicly available at \url{https://github.com/josh1katz/Reflection-invariant-bispectrum}.}

This work underscores the importance of extending MRA methods to non-abelian groups. By doing so, we not only broaden the scope of applicability but also move closer to realistic models encountered in critical fields like cryo-EM. Our findings provide strong motivation for further theoretical and computational exploration of MRA problems associated with other non-abelian groups.

\section{Statement of results}
\subsection{Invariant tensors} \label{sec.moments}
\begin{definition} \label{def.inv}
  Let $V$ be a complex representation of a compact group
  and let $x \in V$ be any vector. The symmetric tensor
  \begin{align}
  T^G_d(x)=
  \int_G (g \cdot x)^{\bigotimes d}\;dg
  \end{align}
  is the degree $d$ {\em polynomial invariant tensor} and
  \begin{align}
M^G_d(x)=\int_G (g \cdot x)^{\bigotimes d-1}\otimes(\overline{g \cdot x})\;dg
  \end{align} is the degree $d$ {\em unitary invariant tensor}.
  (In both cases the integral is with respect to the Haar measure.)
\end{definition}


The polynomial and unitary invariant tensors are the moment tensors for random
group translates of a vector $x \in V$ when the group
elements are taken with respect to a uniform distribution on $G$. For this reason we will also refer to them as moments.
\begin{remark} By definition the polynomial and unitary invariants agree on real vectors. We will take advantage of this fact in the proof of Theorem~\ref{thm.dihedral}.
\end{remark}
\rev{\begin{remark} Throughout this paper we use the term {\em generic} in the sense of algebraic geometry. In particular
we say that a property holds for generic vectors in $\R^n$ (resp. $\C^n$) if there is an open set $U$
whose complement is an algebraic subset of dimension $m < n$ such that the property holds on $U$. Such open
sets are called Zariski open sets and they are dense and have full Lesbegue measure.
\end{remark}
}
\begin{definition} If $V$ is a $n$-dimensional complex representation of a compact group
  then we say that polynomial invariants of degree at most $d$ separate generic
  orbits if there is a non-empty $G$-invariant  Zariski open set $U\subset V$ such that for $x \in U$ the orbit of $x$ is uniquely determined by the invariant
  tensors $T^G_1(x), \ldots ,  T^G_d(x)$. Analogously we say that unitary invariants of degree at most $d$ separate generic orbits if there is a non-empty real Zariski open
  set in $U \subset \C^n = \R^{2n}$ such that for all $x \in U$ the tensors $M^G_1(x), \ldots , M^G_d(x)$ uniquely determine the orbit of $x$.
\end{definition}
\begin{remark}  The condition that polynomial invariants of degree $d$ separate generic orbits implies
that $\Frac (\C[V]^G) = \Frac \C[V^G_{\leq d}]$ where $V_{\leq d}^G \subset \C[V]^G$ is the finite dimensional subspace of invariant polynomials of degree at most $d$.  The converse holds provided that $\Frac \C[V]^G = \C(V)^G$
which is true if $G$ is finite by~\cite[Theorem 3.3]{popov1994invariant}. Note that for a general compact group the invariant ring $\C[V]^G$ is the same as the invariant ring of the corresponding complex algebraic group $G_\C$. 
\end{remark}

\subsection{The standard representation of the dihedral group}
Let $D_{n}$ be the dihedral group of order $2n$ with generators
$r$ of order $n$ and $s$ of order two. Consider the $n$-dimensional
standard representation of $D_{n}$ where the generator $r$ acts by cyclic shifts and the generator $s$ acts by the reflection
$s(x_0, \ldots x_{n-1}) = (x_0, x_{n-1}, x_{n-2}, \ldots x_1)$.

\begin{theorem} \label{thm.dihedral}
  The polynomial invariant tensors $T^{D_n}_1, T^{D_n}_2, T^{D_n}_3$  separate generic complex orbits
  in the standard representation of $D_{n}$.
\end{theorem}

\subsection{Outline of the proof}
The proof of Theorem~\ref{thm.dihedral} proceeds in two steps. We
start by proving that polynomial invariants of degree three separate
\rev{generic} real orbits from each other, Proposition~\ref{prop.real_sep}, and then
use algebro-geometric techniques to bootstrap from real to complex
orbit separation.  To prove Proposition~\ref{prop.real_sep} we work in
the Fourier domain and take advantage of the fact the corresponding
invariants for real vectors are related to the classical power
spectrum and bispectrum which arise signal processing~\cite{tukey1953spectral}.  In the Fourier
domain the standard representation of $D_n$ is the restriction of a
band-limited representation of the orthogonal group $O(2)$, which we
view as the `continuous' dihedral group. We then prove
Theorem~\ref{thm.orthogonal}, a result of independent interest,
stating that for this representation the $O(2)$-invariants of degree
at most three separate generic real orbits. The proof
Proposition~\ref{prop.real_sep} is completed by proving that if two
vectors in an $O(2)$-orbit have the same dihedral invariants then they
lie in a common $D_n$ orbit. 

In Section~\ref{sec.dihedralproofcomplex}
we complete the proof of Theorem~\ref{thm.dihedral} by showing that no non-real vector can have the
same polynomial invariants as a generic real vector. This combined
with the fact that $\R^n$ is Zariski dense in $\C^n$ with respect to
the complex Zariski topology allows us to conclude the proof.
\rev{\begin{remark}The strategy used in Section~\ref{sec.dihedralproof} takes advantage of the fact that the standard representation
of the dihedral group $D_{n}$ extends to a representation of the orthogonal group $O(2)$. In general suppose
$H < G$ and $V$ is a representation of $G$ and thus of $H$. Any $G$-invariant polynomial on $V$
is necessarily $H$ invariant, and any $G$-orbit decomposes into a union of $H$-orbits. 
If we prove that $G$-invariant polynomials of degree at most $d$ separate generic $G$-orbits
then we can prove that $H$-invariant polynomials of degree
at most $d$ separate $H$-orbits if we can find further $H$-invariants of degree at most $d$ which separate a generic $G$-orbit into its $H$-suborbits. This strategy was successful in our case when $G = O(2)$ and $H=D_{n}$, in part because the irreducible
summands in the $O(2)$ representation remain irreducible when we restrict to $D_n$. We are not aware if this strategy
has been previously used in the literature.
\end{remark}}

\section{Proof of theorem~\ref{thm.dihedral} for real vectors} \label{sec.dihedralproof}
The goal of this section is to prove the following proposition.
\begin{proposition} \label{prop.real_sep}
  There is a real Zariski open set $U \subset \R^n \subset \C^n$ such that for
  all $x \in U$ the following condition holds. If  $x' \in \R^n$
  and $T^{D_n}_d (x) = T^{D_n}_d (x')$ for $d = 1,2,3$ then $x' = gx$
  for some $g \in D_n$.
\end{proposition}
To simplify many of the computations we work in the Fourier basis
and identify
$\mathbb{R}^n$ with its image in $\mathbb{C}^n$ under the discrete Fourier transform. The image of $\R^n$ under the discrete Fourier transform is
the set of vectors $(f[0], \ldots , f[n-1]) \in\C^n$ satisfying
the condition $f[i] = \overline{f[n-i]}$. In particular we will identify 
a vector $x \in \R^n$ with its Fourier transform
$f= (f[0], f[1], \ldots , f[n-1]) \in \C^n$.

Let $r$ and $s$ be the rotation and reflection which generate $D_{n}$. In the Fourier basis the action of $r$ can be diagonalized 
\begin{align}
    r(f[\ell])=e^{\frac{2\pi \iota \ell}{n}}f[l]
\end{align}
and the reflection action is given by 
\begin{align} 
     s(f[\ell])=f[n-\ell]=\overline{f{[\ell]}}.
\end{align}
The standard representation $\C^n$ decomposes as a sum
$L_0 + V_1 \ldots + V_{(n-1)/2}$
if $n$ is odd
and
$L_0 + V_1 \ldots + V_{n/2 -1} +L_{-1}$ if $n$ is even.
Here $L_0$ is the trivial
representation, $L_{-1}$ the character where the rotation acts with weight $-1$ and reflection trivially. The representation $V_\ell$ is the two-dimensional representation
of $D_n$ where $r$ acts with weights $e^{\pm 2\pi \iota \ell/n}$ and
the reflection exchanges the eigenspaces.

The irreducible two-dimensional dihedral representation $V_\ell$ is the restriction of an irreducible two-dimensional $O(2)$ representation where the rotation $e^{\iota \theta} $ acts with weights $e^{\pm  \iota \ell \theta}$ and
the reflection exchanges the eigenspaces. Let $W_{k} = 
\oplus_{\ell =1}^k V_\ell$.
Before addressing the dihedral case we analyze the $S^1 =SO(2)$ and $O(2)$ orbit recovery problems for
(the Fourier transform of) real vectors in the representation $W_k$.

\subsection{$SO(2)$ and $O(2)$-invariants}
\begin{theorem} \label{thm.orthogonal}
Let $W_k = \oplus_{\ell =1}^k V_k$ be the $2k$-dimensional real representation of $O(2)$.
\begin{enumerate}
\item The invariants of degree two and three for the restricted action of $SO(2)$ separate generic real orbits in $W_k$.
\item If $k \neq 3$ the invariants of degree two and three \rev{for the action of $O(2)$} separate
generic real orbits in $W_k$.
\end{enumerate}
\begin{remark} Part (1) of Theorem~\ref{thm.orthogonal} is well-known in the signal processing community and this result is usually refered to as the statement that the bispectrum recovers orbits of band-limited functions - see for example~\cite{ma2019heterogeneous}. However part (2) for the orthogonal group $O(2)$
  appears to be new.
\end{remark}  
\end{theorem}
To establish notation let $e_{-k}, \ldots , e_{k}$ be an orthonormal basis for $W_{k}$ in the Fourier domain
where
$e^{\iota \theta} \in SO(2)$ acts by $e^{\iota \theta} \cdot e_m = 
e^{\iota m \theta}e_m$. With respect to this basis the irreducible $O(2)$-representation $V_\ell$ is
spanned by $e_\ell, e_{-\ell}$. If we expand a vector
$f \in W_k$ in terms of this basis as
\begin{equation} \label{eq.aij}
f = \sum_{\ell =-k}^k f[\ell]e_{\ell}
\end{equation}
then the fact that $f$ is the Fourier transform of a real vector
implies that $f[-\ell] = \overline{f[\ell]}$.

\subsubsection{Proof for $SO(2)$-invariants}
With this notation we have
\begin{align}
T_2^{SO(2)}(f)=\sum_{i=1}^k {f}[i]{f}[-i]e_ie_{-i}=\sum_{i=1}^k|{f}[i]|^2e_ie_{-i}
\end{align} 

\begin{align}
  T_3^{SO(2)}(f)=\sum_{\{(i,j)| -k \leq i+ j \leq k\}}{f}[i]{f}[j]{f}[-i-j]e_ie_je_{-i-j}.\\
   = \sum_{\{(i,j)| -k \leq i+ j \leq k\}}{f}[i]{f}[j]\overline{{f}[i+j]}e_ie_je_{-i-j}.\\ 
\end{align}
Note that the coefficients of the 
degree-two invariant tensor determine
the componentwise absolute value (power spectrum) of the vector $f$.
The coefficients of $T^{SO(2)}_3(f)$, $\{ f[i]f[j] \overline{f[i+j]}\}$
are known as the {\em bispectrum} in the signal processing literature.

We will show that if the Fourier coefficients $f[\ell]$ are non-vanishing for $1 \leq \ell \leq k$ (and hence for $-k \leq \ell \leq -1$) then
the degree-three invariant tensor determines 
a set of equations for the phases of the Fourier coefficients
which determines them up to the action of $SO(2)$. 
Writing
$f[j]=|{f}[j]|e^{i\theta_j}$ then, given the power spectrum,
the third moment determines for all pairs
$(i,j)$ with $i,j \geq 1$ and $i + j \leq k$
\begin{equation}
    a_{ij} = e^{\theta_i+\theta_j-\theta_{i+j}}
\end{equation}
on the unit circle.
\begin{proposition} \label{prop.so2}
For a vector $f\in \R^{2k}$ with $f[i]\neq 0$  for $1\leq i\leq k$, 
the 
$a_{ij}$ defined in~\eqref{eq.aij} determine the unknown phases 
up to translation by an element of $SO(2)$.
\end{proposition}
\begin{proof}
From the $a_{ij} =\{e^{i(\theta_i+\theta_j-\theta_{i+j})}\}$ 
we can recursively express the phase $e^{i \theta_\ell}$ for $\ell > 1$ in
terms of $e^{i \theta_1}$. 
Precisely we have for $\ell \geq 2$
$$e^{i\theta_\ell}=
a_{1,\ell-1}^{-1} e^{i \theta_1} e^{i \theta_{\ell-1}}.$$
\end{proof}
\subsubsection{$O(2)$ Invariants} 
Now for the orthogonal case. The degree-two invariant tensor
$$T_2^{O(2)}(f)=\sum_{0\leq i\leq k} {f}[i]{f}[-i]e_ie_{-i}$$
is the same as the degree-two $SO(2)$-invariant tensor and therefore gives us the power spectrum of $f$.
The difference arises when we look at the third moment.
\begin{equation} 
\begin{split}   
T_3^{O(2)}(f)
&=\sum_{(i,j) \in S}({f}[i]{f}[j]{f}[-i-j]+{f}[-i]{f}[-j]{f}[i+j])(e_ie_je_{-i-j}+e_{-i}e_{-j}e_{i+j})\\
&=\sum_{(i,j)\in S}({f}[i]{f}[j]{f}[-i-j]+\overline{{f}[i]{f}[j]{f}[-i-j]})(e_ie_je_{-i-j}+e_{-i}e_{-j}e_{i+j})
\end{split}  
\end{equation}
where $S = \{(i,j)\in [1,k]^2| i+j \leq k\}$
Assuming all of the $f_\ell$ are non-vanishing
the third moment then determines for each pair of indices $(i,j)$
with $i,j \geq 1$ and $i+j \leq k$
\begin{equation}
\alpha_{i,j}(f) =\cos(\theta_i+\theta_j-\theta_{i+j})
\end{equation}
Since cosine is an even function,
the quantity $\alpha_{i,j}$ determines the angle $(\theta_i + \theta_j -\theta_{i+j})$ up to sign or, equivalently the complex
number $a_{i,j}(f) = e^{\iota (\theta_i + \theta_j - \theta_{i+j})}$ up to complex
conjugation. Precisely, $a_{i,j}(f) = \alpha_{i,j} + i\sqrt{1-\alpha_{i,j}^2}$ or $a_{i,j}(f) = \alpha_{i,j} - i\sqrt{1-\alpha_{i,j}^2}$.
The set of complex numbers $\{a_{i,j}(f)\}_{(i,j)}$ is invariant under
rotation of $f$ by an element of $SO(2)$ and if $r \in O(2) \setminus SO(2)$ 
then $a_{i,j}(r\cdot f) = \overline{a_{i,j}(f)}$.

For a fixed vector $f$, consider the set of real vectors $$f' = \sum_{\ell = 1}^k f'[-\ell]e_{-\ell} + f'[\ell]e_{\ell} \in \R^{2k}$$ that satisfy the equations
$$T_2^{O(2)}(f') = T_2^{O(2)}(f)$$ and 
$$T_3^{O(2)}(f') = T_3^{O(2)}(f).$$
Since $f$ and $f'$ have the same power spectrum we know that 
$f'[\ell] = |f[\ell]| e^{\iota \varphi_\ell}$ for some
angle $\varphi_\ell$. 
Equality of the degree-three
invariants implies that for each pair of indices $(i,j)$ with $i+j \leq k$
$\alpha_{i,j}(f') = \alpha_{i,j}(f)$ where
$\alpha_{i,j}(f') = \cos(\varphi_i + \varphi_j - \varphi_{i+j})$.
The next result shows that under a suitable genericity hypothesis
on the Fourier coefficients of the original vector $f$, a single choice of 
the complex number $a_{p,q}(f') = e^{\iota(\varphi_p + \varphi_q - \varphi_{p+q})}$ determines all other $a_{i,j}(f')$ uniquely from
the degree-three invariants. Because the vector $f'$ is real we 
know that $a_{-i,-j}(f') = \overline{a_{i,j}(f')}$ 
so we need only determine the $a_{i,j}(f')$ with $i,j \geq 1$.
This in turn will imply that the set of $f' \in \R^n$ which satisfy the equations
$T_2^{O(2)}(f') = T_2^{O(2)}(f)$ and $T_3^{O(2)}(f') = T_3^{O(2)}(f)$ 
form a single $O(2)$ orbit.

\begin{lemma} \label{lem.star}
The set of vectors $f = \sum_{\ell =1}^k f[-\ell]e_{-\ell} + f[\ell]e_\ell$ with non-zero Fourier coefficients which satisfies the following condition.\\
(*) There is no non-empty set of indices $S \subset \{(i,j)\subset [1,k]^2| i+j \leq k\}$
and all non-zero integers $n_{i,j}$ such that 
\begin{equation} \label{eq.star}
\prod_{(i,j)\in S} \left(\overline{a_{i,j}(f)}\right)^{n_{i,j}}=\prod_{(i,j)\in S} a_{i,j}(f)^{n_{i,j}}
\end{equation}
is a non-empty $O(2)$-invariant Zariski open set in $\R^{2k}$
\end{lemma}
\begin{proof}
We must show that the complement of the set of vectors satisfying
(*) is Zariski closed in the set of vectors with non-zero Fourier coefficients. Given a non-empty set $S$
of indices if a vector $f$ satisfies~\eqref{eq.star} then multiplying
both sides by $|f[i] f[j] f[i+j]|$ it also satisfies the non-zero
real polynomial equation in its Fourier coefficients.
\begin{equation}
\prod_{(i,j)\in S} \left(f[i]f[j]f[-i-j]\right)^{n_{i,j}}=\prod_{(i,j)\in S} \left(f[-i]f[-j]f[i+j]\right)^{n_{i,j}}
\end{equation}
and hence lies in a proper real subvariety $X_S \subset \R^{2k}$.
Thus, the set of vectors $f \in \R^{2k}$ with non-zero Fourier coefficients satisfying (*) is the Zariski open set
$\R^n \setminus \left(\bigcup_S X_S\right)$.
\end{proof}
\begin{proposition} \label{prop.crux}
Suppose that for a pair of indices $(p,q)$ with $p + q \leq k$ we choose one of the two possible values for $a_{p,q}(f')=e^{i(\varphi_t+\varphi_p-\varphi_{p+q})}$ then, if $k \geq 4$, all other $a_{i,j}(f')$ are uniquely determined 
from the corresponding real numbers $\alpha_{i,j}$ provided
that the vector $f$ satisfies condition (*) of Lemma~\ref{lem.star}
\end{proposition}
\rev{We note that condition (*) is a sufficient condition for the other $a_{i,j}(f')$ to be uniquely determined from the choice
of a single $a_{p,q}$, but we do not know if it is necessary.}
The proof of Proposition~\ref{prop.crux} requires the following
combinatorial proposition.
\begin{proposition} \label{prop.annihilator}
If $k \geq 4$, then there exist  integers $m_{i,j}$ all non-zero such that 
for any $(t_1, \ldots , t_k) \in \R^k$
$$\sum_{\{(i,j)\geq 1|i+j \leq k\mathbb\}
} m_{i,j} (t_i +t_j -t_{i+j}) = 0$$
\end{proposition}
We now prove Proposition~\ref{prop.crux} assuming Proposition~\ref{prop.annihilator}.
\begin{proof}[Proof of Proposition~\ref{prop.crux}]
After possibly replacing $f$ with $r \cdot f$ where $r \in O(2)$ we can assume $a_{p,q}(f') = a_{p,q}(f) = \alpha$.
Then we know that for every other pair of indices $(i,j)$ with $i + j\leq k$,
$a_{i,j}(f') = a_{i,j}(f)$ 
or $a_{i,j}(f') = \overline{a_{i,j}}(f)$.
Let $S$ be the set indices $(i,j)$ such that $a_{i,j}(f') =
\overline{a_{i,j}(f)}$. By assumption $(p,q) \notin S$
so $S$ is a proper subset of the indices $(i,j)$.
By Proposition~\ref{prop.annihilator} applied to the
angles $\theta_i$ and $\varphi_i$ respectively we know that there are 
integers $m_{i,j}$ all non-zero such that
\begin{equation} \label{eq.cancel} \prod_{(i,j)} a_{i,j}(f')^{m_{i,j}} =\prod_{(i,j)}a_{i,j}(f)^{m_{i,j}}=1.
\end{equation}
Since $a_{i,j}(f') = a_{i,j}(f)$ for $(i,j) \in S^c$
we can reduce~\eqref{eq.cancel} to 
\begin{equation}
\prod_{(i,j) \in S}a_{i,j}(f')^{m_{i,j}} = \prod_{(i,j) \in S}a_{i,j}(f)^{m_{i,j}}.
\end{equation}
Since $a_{i,j}(f') = \overline{a_{i,j}(f)}$ for $(i,j) \in S$
this means that $f$ does not satisfy hypothesis (*) unless $S = \emptyset$. 
Thus $\alpha_{i,j}(f') = \alpha_{i,j}(f)$ for all pairs $(i,j)$.
\end{proof}
\begin{proof}[Proof of Proposition~\ref{prop.annihilator}]
To prove Proposition~\ref{prop.annihilator} it suffices to prove 
if $x_1, \ldots , x_k$ is the dual basis for $(\R^k)^*$ then the forms
$x_{ij} = \{x_i +x_j -x_{i+j}\}$ satisfy an equation of linear dependence 
$$\sum_{ij}q_{ij}x_{ij}$$ with $q_{ij} \neq 0$ and rational for all pairs $(i,j)$. Since the $x_{ij}$ are integral linear combinations of
the basis vectors $x_{i}$, we can view them as being elements
of $(\Q^k)^*$ and the result will follow from a few linear algebra
lemmas.
\begin{lemma} \label{lem.nullspace}
Let $V$ be a vector space over an infinite field $K$ and 
let $v_1, \ldots v_m$ be a collection of non-zero vectors in $V$ and let $W=\{(w_1,...,w_m)\}\subset K^m$ 
be the space of solutions to the equation 
$\sum_{i=1}^m w_iv_i=0$. Assume that for each $j$ there is a solution vector
$(w_1,...,w_m)\in W$ with $w_j\neq 0$. Then there is a solution $(w_1,...,w_m) \in W$ with $w_j\neq0$ for all $j$.
\end{lemma}
\begin{proof}
By assumption  $W$ is not contained in the hyperplane
$V(x_j)$
for each $j$. Hence, since $W$ is irreducible (because it is a linear subspace) it cannot be contained in the union of the $V(x_j)$
which implies that there exists a solution vector $(w_1, \ldots , w_m) \in K^m$
with all $w_j$ non-zero.
\end{proof}
\rev{
\begin{definition}
Let $V$ be a vector space and 
let $S= \{v_1...,v_m\}\subset V$ be a spanning set with $m>\dim V$. 
If for all $i\leq m$ the set $S_i=\{v_1...,\hat{v_i},...,v_m\}$, where we delete the vector $v_i$, still spans $V$ 
we say that $V$ is an {\em excessive spanning set.}
\end{definition}
\begin{lemma} \label{lem.excessive}
Let $V$ be a vector space over an infinite field $K$ and 
let $S =\{v_1...,v_m\}\subset V$ be an excessive spanning set.
Then there is an equation of linear dependence 
$\sum_{i=1}^m w_iv_i=0$ with $w_i\neq 0$ for all $i$. 
\end{lemma}
}
\begin{proof}
By Lemma~\ref{lem.nullspace}
it suffices to find, for each $j$, a linear dependence $\sum_{i=1}^kw_iv_i=0$ with $w_j\neq 0$. 
By assumption $S_j$ spans $V$ so $v_j=\sum_{i\neq j} w_iv_i$ 
and hence the equation $\sum_{i=1}^k w_iv_i=0$ holds with $w_j=-1$. 
\end{proof}

Proposition~\ref{prop.annihilator} now follows from the following lemma.
\end{proof}
\begin{lemma} \label{lem.hyperplane}
Let $V^-_k$ be the hyperplane in $(\Q^k)^*$ annihilated
by the vector $(1,2,\ldots , k) \in \Q^k$. 
If $k\geq 2$ then the set of linear forms $\{x_{ij} =x_i+x_j-x_{i+j}\}\subset (\Q^k)^*$ 
spans the hyperplane $V^-_k$ 
and if $k \geq 4$ it forms an excessive spanning set.  
\end{lemma}
\begin{proof}
First we note that
$\langle x_{ij},(1,2, \ldots , k)\rangle = i + j -(i+j) =0$, so the
span of the $x_{ij}$ lies in $V^-_k$.
Let $L_k$ be the subspace of $(\Q^k)^*$ spanned by the $\{x_{ij}\}$.
We will prove by induction on $k$ that $L_k = V^-_k$.
To start the induction note that if $k =2$ then $V^-_2$ is spanned by $x_{11} =2x_1 -x_2$. 
To obtain the induction step observe that
the embedding $$(\Q^{k})^* \subset (\Q^{k+1})^*,
(l_1, \ldots ,l_k) \mapsto (l_1, \ldots, l_k, 0)$$ embeds $V^-_{k}$ 
into $V^-_{k+1}$. Assume by induction that the $V_k^-$ is spanned by $\{x_{ij}\}_{i+j \leq k}$; i.e. $V^-_k = L_k$. Then
$L_{k+1}$ contains the additional linear form
$x_1 + x_k - x_{k+1}$ which is not in $L_k$. Hence
$\dim L_{k+1} \geq \dim L_k + 1 = k$. On the other hand, we know that $L_{k+1}$ is contained in the hyperplane $V^-_{k+1}$. Therefore,
$L_{k+1} = V^-_{k+1}$ since they are both hyperplanes.

A similar induction will allow us to show that if $k \geq 4$
then the $\{x_{ij}\}_{i+j \leq k}$ is an excessive spanning set of
$V^-_k$. If $k=4$ then any three of the forms 
$$2x_1 -x_2, x_1+x_2 - x_3, x_1 +x_3-x_4,2x_2 - x_4$$ are linearly independent so they
form an excessive spanning set for the hyperplane $V^-_4$.
Assume by induction that $S_k= \{x_{ij}\}_{i+j \leq k}$ is
an excessive spanning set for $V^-_k$. If we let
$S_{k+1} = \{x_{ij}\}_{i+j \leq k+1}$ then $S_{k+1} = S_k \coprod
T_{k+1}$ where $T_{k+1} = \{x_{ij}\}_{i+j = k+1}$.
If we remove a vector $x_{ij} \in S_k \subset S_{k+1}$
then we know that $S_k \setminus \{x_{ij}\}$ spans
$V^-_k$ by induction. Since none of the vectors
in $T_k$ is contained in $V^-_k$ we conclude
that $S_{k+1} \setminus \{x_{ij}\}$ spans all of $V^-_{k+1}$. On the
other hand if $x_{i,k+1-i}$ is any vector in $T_{k+1}$ then
$V^-_{k+1}$ is spanned by $S_k \cup \{x_{i,k+1-i}\}$
since $x_{i,k+1-i} = x_i + x_{k+1-i} -x_{k+1}$.
Thus, as long as $T_{k+1}$ contains at least two vectors, which
is the case if $k \geq 4$,
then we may delete any vector from $S_{k+1}$ and still obtain
a spanning set for $V^-_{k+1}$.
\end{proof}

\begin{proof}[Proof of Theorem~\ref{thm.orthogonal} for $O(2)$]
We first consider the case where $k > 3$.
Given a vector $f$ satisfying the generic condition (*) of Lemma~\ref{lem.star}, suppose that $f'$ has the same degree two and
three invariant moments. Then $f$ and $f'$ have the same power spectrum.
If $k >3$ and $f$ satisfies condition (*) then Proposition~\ref{prop.crux} implies that after possibly replacing $f$ with $r \cdot f$ for some $r\in O(2)$ we have 
$a_{ij}(f') = a_{ij}(f)$ for all pairs of indices $(i,j)$. By Proposition~\ref{prop.so2} this implies that $f'$ and $f$ are in the
same $SO(2)$ orbit after applying a reflection. Hence $f'$ and $f$
are in the 
same $O(2)$ orbit.

If $k =1$ then $V_1$ is the defining representation of $O(2)$ and 
the second moment $f[-1]f[1] = |f[1]|^2$ uniquely determines the orbit.
If $k =2$ and then the second moment of a vector $f$ determines
the power spectrum of $f$ and $\cos(2\theta_1 - \theta_2)$. Rotating
by an element of $SO(2)$ we can make $\theta_1$ arbitrary and for
each choice of $\theta_1$ there are two possible values
for $\theta_2$ which preserve the quantity $\cos(2\theta_1 - \theta_2)$.
Hence, if $f$ has non-vanishing Fourier coefficients,
the set of vectors with same degree two and three invariants consists of a single $O(2)$ orbit.
\end{proof}

\begin{example} \label{k=3}
When $k=3$ the invariants of degree at most three recover the
$O(2)$ orbit of a generic vector up to a list of size two. 
If $f$ and $f'$ have the same invariants of degree at most three
then, after possibly applying a reflection, we can assume $a_{1,1}(f') = a_{1,1}(f)$, and after applying a rotation that $e^{\iota \varphi_1} = e^{\iota \theta_1}$.
Since $a_{1,1}(f') = e^{2 \iota \varphi_1 + \varphi_{-2}} =
a_{1,1}(f) = e^{2 \iota \theta_1 + \theta_{-2}}$.
This implies that $\varphi_{-2} = \theta_{-2}$, and hence
that $\varphi_{2} = \theta_{2}$ since $f'$ and $f$ are the Fourier transforms of real-valued functions. The only other piece of information we have is that 
\begin{equation} \label{eq.uncertain}
\cos(\varphi_1 + \varphi_2 -\varphi_3)=
\cos(\theta_1 + \theta_2 - \theta_3).
\end{equation}
However, for a fixed angle $\theta_3$ there 
are two angles $\varphi_3$ for which~\eqref{eq.uncertain}
holds.
\end{example}
\begin{remark}
It is an interesting question as to the computational difficulty of recovering a vector from the orthogonal bispectrum. If we attempt to do frequency marching as in~\cite{bendory2017bispectrum} we are forced to search over a set of $2^{n-1}$ conjugate possibilities before we determine the correct one. This is probably not optimal. 
\end{remark}
\subsection{Recovering real vectors from cyclic invariants of low degree} \label{sec.cyclic}
The standard representation $\mathbb{C}^n$ of $\Z_n$ is the 
regular representation and by~\cite{bandeira2023estimation, edidin2025orbit}
we
know that the invariants of degree at most three separate generic real and complex orbits. However, we will analyze this case directly as it will aid us in our proof that the dihedral invariants of degree at most three also separate generic dihedral orbits.

\noindent {\bf Case I. $n$ odd.} If $n =2k+1$ is odd, then
the standard representation decomposes as $L_0 + W_k$ where 
$W_k = \oplus_{\ell =1}^k V_\ell$ and $L_0$ is the trivial representation. If we expand $f \in \R^n$ as $f[0]f_0 + \sum_{\ell =1}^k
f[-\ell]e_{-\ell} + f[\ell]e_\ell$ with $f[-\ell] = \overline{f[\ell]}$ then the degree one invariant
gives $f[0]$. We also know that the $SO(2)$-invariants (which
are necessarily $\Z_n$ invariants)
in degree three determine the $f[\ell]$ up to multiplication
by $e^{\iota \ell \theta}$ for some angle $\theta$ which is independent of $\ell$. However if we choose $p,q,r >0$ so that $p+q + r =n$
then $f[p]f[q]f[r]$ is also $\Z_n$-invariant which implies
that $e^{\iota n \theta}=1$; i.e., the ambiguity is only up to to translation by an
element of $\Z_n$.

\noindent{\bf Case II. $n$ is even.}
If $n = 2k+2$ is even then the standard representation decomposes
as $L_0 + L_{-1} + W_k$, where $L_0$ is the trivial representation and $L_{-1}$ and is the character where the generator of $\Z_n$ acts by multiplication by $-1$. If we expand $f \in \R^n$ as
$f= f[0]f_0 + h[-1]f_{-1} + \sum_{\ell =1}^k f[-\ell]e_{-\ell} + f[\ell]e_{\ell}$ then
once again the degree-one invariant is $f[0]$ and the
same argument used in the previous case implies that the degree two and three invariants determine the Fourier coefficients $f[\ell]$
for $l =1,\ldots k$ up to the action of $\Z_n$. The Fourier 
coefficient $h[-1]$ is determined from the $f[\ell]$ and the degree-three
$\Z_n$ invariant function $h[-1]f[1]f[n/2-1]$.
\subsection{Dihedral invariants}
We now follow the same approach for the dihedral invariants. The proofs
for $n=2k+1$ and $n=2k+2$ are essentially identical as was the case for 
the cyclic invariants. 

{\bf Case I. $n=2k+1$, $k \neq 3$.}
Suppose that $T_d^{D_n}(f)$ is known for $d \leq 3$.
The Fourier coefficient $f[0]$ is once again known from the single
degree-one invariant.
By Theorem~\ref{thm.orthogonal} we know that, after possibly
replacing $f$ with its translate by the reflection $s \in D_n$,
the subset
of the $D_n$ invariants which are also $O(2)$ invariants
determine the Fourier coefficients $f[\ell]$ up to multiplication by $e^{\iota \ell \theta}$ for some angle $\theta$. If $p+q +r =n$, then $D_n$-invariance of
$f[p]f[q]f[r] + f[-p]f[-q]f[-r]$ implies that $e^{\iota n \theta} =1$,
so $f$ is determined up to cyclic shift and rotation; i.e., by an element of 
$D_n$.

{\bf Case II. $n=2k+2$, $k\neq 3$.}
The same argument used in the odd case shows that the $D_{n}$
invariants of degree at most three determine $f[0]$
and the Fourier coefficients $f[\ell]$ up to translation by an element
of $D_n$. The coefficient $h[-1]$ is determined from the other Fourier coefficients by the degree-three invariant $h[-1](f[1]f[n/2-1]+f[n-1]f[n/2+1])$.

Because Theorem~\ref{thm.orthogonal} requires $k\neq 3$ we need 
an additional proposition to cover the standard representations
of the dihedral groups $D_7$ and $D_8$.
\begin{proposition}
Let $\mathbb{R}^{7}$ and $\mathbb{R}^8$ be the standard representations of $D_{7}$ and $D_8$ respectively. The $D_n$-orbit
of a generic vector $f\in \mathbb{R}^{n}$ is determined by the 
invariants of degree at most three.
\end{proposition}
\begin{proof}
We will prove the result for $D_7$ but a similar argument will also work for $D_8$. Let $Z \subset \R^7$ be the $D_7$-invariant proper Zariski closed
subset of vectors whose
Fourier coefficients $f[\ell]$ satisfy the relation among $\Z_7$-invariants
\begin{equation}\label{eq.deltas}
\begin{split}
(f[1]f[2]f[-3])^{\delta_1} (f[-2]^2f[-3])^{\delta_2} (f[1]f[3]^2)^{\delta_3} =\\ (f[-1]f[-2]f[3])^{\delta_1} (f[2]^2f[3])^{\delta_2} (f[-1]f[-3]^2)^{\delta_3}
\end{split}
\end{equation}
for some $\delta_1, \delta_2, \delta_3 \in \{0,1\}$ not all zero.
We will prove that for any vector $f \in \R^7 \setminus Z$ if
$f' \in \R^7$ satisfies
$T_d^{D_n}(f') = T_d^{D_n}(f)$ for $d =1,2,3$ then 
$f' = g\cdot f$ for some $g \in D_7$.

Equality of the degree-one invariants immediately implies the equality of Fourier coefficients $f'[0] = f[0]$ corresponding to the trivial summand in $\R^7$. Equality of the second moments
implies that $f'$ and $f$ have the same power spectra.
Thus if $f[\ell] = r_\ell e^{\iota \theta_\ell}$ for some angles
$\theta_1, \theta_2,\theta_3$ then
$f'[\ell] = r_\ell e^{\iota \varphi_\ell}$ for some angles $\varphi_1,\varphi_2, \varphi_3$.

Equality of the degree-three invariants
$$\begin{array}{lll}
f'[1]^2f'[-2] + f[-1]f'[2]^2 &=& f[1]^2f[-2]+f[-1]^2f[2]\\
f'[1]f'[2]f'[-3] + f'[-1]f'[-2]f'[3] & =&f[1]f[2]f[-3]+f[-1]f[-2]f[3]\\ f'[1]f'[3]^2 + f'[-1]f'[-3]^2 & = &f[1]f[3]^2+f[-1]f[-3]^2\\
f'[2]^2f'[3] + f'[-2]^2f'[-3] & = &f[2]^2f[3]+f[-2]^2f[-3]
\end{array}$$
implies that
$$\alpha_{i,j}(f') = \alpha_{i,j}(f)$$
where $\alpha_{i,j}(f) = \cos(\theta_i + \theta_j - \theta_{i+j})$,
$\alpha_{i,j}(f') = \cos(\varphi_i + \varphi_j -\varphi_{i+j})$
and all indices are taken $\bmod \;7$.
After possibly applying a reflection in $D_7$ we can assume that
$a_{1,2}(f') = a_{1,2}(f) = \alpha$ where 
$a_{i,j}(f) = e^{\iota (\theta_i + \theta_j - \theta_{i+j})}$,
$a_{i,j}(f') = e^{\iota (\varphi_i + \varphi_j - \varphi_{i+j})}$
and again all indices are taken $\bmod \;7$.

The relation on the exponential expressions
\begin{align*}
e^{\iota(\varphi_1+\varphi_2-\varphi_3)}e^{-\iota(2\varphi_2+\varphi_3)}e^{\iota(\varphi_1+2\varphi_3)}=e^{\iota(2\varphi_1-\varphi_2)}=\alpha
\end{align*}
implies that
\begin{equation}
a_{1,2}(f')a_{-2,-2}^2(f')a_{1,3}(f')=\alpha=a_{1,2}(f)a_{-2,-2}^2(f)a_{1,3}(f).  
\end{equation}
Let $S$ be the set of pairs $(i,j)$ such 
$a_{i,j}(f') = \overline{a_{i,j}(f)}$. Then for $(i,j) \in S$
we have that $a_{i,j}(f') = a_{-i,-j}(f)$. If $S$ is non-empty
then after multiplying by
the moduli of the Fourier coefficients we obtain an equation
on the Fourier coefficients $f[\ell]$ of the form~\eqref{eq.deltas} with not all $\delta_i = 0$. Thus if $f$ is generic then
$S = \emptyset$; ie. $a_{i,j}(f') = a_{i,j}(f)$.
Hence by the reasoning of Section~\ref{sec.cyclic} we conclude
that $f' = r \cdot f$ for some rotation in $D_{7}$. Hence  $f'$ is obtained from $f$ after applying a possible reflection
and rotation; i.e., by an element of $D_7$.
\end{proof}

\section{Separation of complex dihedral orbits} \label{sec.dihedralproofcomplex}
Thus far we have proved that the $D_n$-orbit of a generic vector
$x \in \R^n$ is determined by its first three invariant tensors.  The argument made crucial use of the fact that degree-two dihedral invariants determine the power spectrum of a real vector. This is no longer the case for vectors $f \in \C^n$. To prove that generic complex orbits are separated by these invariants, we need an additional argument.
By the following Proposition it suffices to show that if $f \in \R^n$
then any {\em complex} vector $f'$ with the same invariants of degree at most three must also be real.

\begin{proposition} \label{prop.realtest}
Let $G$ be a finite group acting on a real vector space $V = \R^n$,
and let $T\subset \R[V]^G$ be a set of invariants which separates the generic real orbits of $V$. Assume in addition the following property holds:\\

(*) For generic $x\in \R^n$ any $y \in \C^n = V \otimes \C$ 
with the property that $p(y) = p(x)$ for all $p \in T$ implies that
$y \in \mathbb{R}^n$\\

Then $\Frac(\C[T]) = \Frac(\C[V]^G)$ and consequently the invariants in $T$
separate the generic orbit of any $x \in \C^n = V \otimes \C$.
\end{proposition}
\begin{proof}
  The inclusion $\C[T] \subset \C[V]^G \subset \C[V]$ induces dominant maps of
  irreducible varieties $\C^n \stackrel{\pi} \to \C^n/G \stackrel{f} \to \Spec \C[T]$.
  By hypothesis, there is a real Zariski dense set of points $y \in \R^n$ such that
  the fibers of $f \circ \pi$ consists of $G$-orbits. Since $\R^n$ is Zariski dense in $\C^n$ in the complex Zariski topology there is a Zariski dense
  set of points $X \subset \C^n$ such that for all $x \in X$, $ (f\circ \pi)^{-1} (f \circ \pi)(x) = \pi^{-1}(\pi(x))$. In other words, the dominant map
  $f \colon \C^n/G \to \Spec \C[T]$ is bijective on the Zariski dense open
  set $\pi(X) \subset \C^n/G$. Hence the map $f$ is birational so
  $\Frac( \C[(T]) = \Frac( \C[V]^G)$ and thus there is a Zariski dense set
  in $\C^n$ of points whose orbits are uniquely determined by the invariants in $U$.
\end{proof}

Following our general proof strategy we will begin by proving this for cyclic invariants, even though this is known~\cite{bandeira2023estimation, edidin2025orbit}.
\subsection{Cyclic group}
\begin{lemma} \label{lem.complexcyclic}
Let $f \in \R^n$ be a vector with all Fourier coefficients $f[\ell]$
non-vanishing. Suppose that $f' \in \C^n$ satisfies 
$T^{\Z_n}_d(f')=T^{\Z_n}_d(f)$ for $d\leq 3$ then $f'\in \mathbb{R}^{n}$; i.e., $f'[-i]=\overline{f'[i]}$ for all $i$. 
\end{lemma}
\begin{proof}
First note that $f'[0]=f[0]\in \mathbb{R}$ by assumption and if $n$ is even $f'[n/2]^2=f[n/2]^2\in \mathbb{R}_{>0}$ which implies $f'[n/2]\in \mathbb{R}$.
In general we have $f'[i]f'[-i] = |f[i]|^2\in \mathbb{R}_{\geq 0}$ which implies that if $f'[i]=r'_ie^{\iota\varphi_i}$ then 
$f'[-i]=r'_{-i}e^{-\iota\varphi_{i}}$. It remains to show that $r'_i=r'_{-i}$ where the indices are taken $\bmod \; n$.

Since $f$ is a real vector we know that $f[-i]= \overline{f[i]}$
for all $i$. In particular for all pairs $(i,j)$ we have the equality
of degree-three $\Z_n$ invariants
$$f[i]f[j]f[-i-j]=\overline{f[-i]f[-j]f[i+j]}$$
where the indices are taken modulo $n$.
Since $f'$ and $f$ have the same degree-three $\Z_n$-invariants 
we can also conclude that 
$$f'[i]f'[j]f'[-i-j]=\overline{f'[-i]f'[-j]f'[i+j]}$$
for all pairs $(i,j)$ with the indices taken $\bmod \;n$.
This in turn 
implies that
\begin{equation}\label{eq.rij}
r'_ir'_jr'_{-i-j}=r'_{-i}r'_{-j}r'_{i+j}.
\end{equation}
In particular 
$r'_2=(\frac{r'_1}{r'_{-1}})^2r'_{-2}$. Substituting this into the relation $r'_1r'_2 r'_{-3} = r'_{-1}r'_{-2}r'_{-3}$
yields the relation $r'_3=(\frac{r'_1}{r'_{-1}})^3r'_{-3}$ and in general 
\begin{equation}\label{eq.magnituderel}
    r'_t=\left( \frac{r'_1}{r'_{-1}}\right)^tr'_{-t}.
 \end{equation}
for $t \leq n/2$, so it suffices to prove that $r'_1/r'_{-1} =1$.

On the other hand if $i + j > n$, 
then substituting~\eqref{eq.magnituderel} into~\eqref{eq.rij} yields
\begin{align}  
r'_ir'_jr'_{n-i-j}=(\frac{r'_1}{r'_{-1}})^ir'_{-i}(\frac{r'_1}{r'_{-1}})^jr'_{-j}(\frac{r'_1}{r'_{-1}})^{n-i-j}r'_{-n+i+j}\\
=(\frac{r'_1}{r'_{-1}})^nr'_{-i}r'_{-j}r'_{-n+i+j}=r'_{-i}r'_{-j}r'_{-n+i+j} \end{align}
which yields $(\frac{r'_1}{r'_{-1}})^n=1$. Since we also know that $r_1, r_{-1}$ are both positive we conclude that $r'_1 = r'_{-1}$.
\end{proof}
\subsection{Dihedral group}
We conclude the proof of Theorem~\ref{thm.dihedral} by proving a similar lemma for the dihedral group.
\begin{lemma}
Let $f \in \R^n$ be a vector with all Fourier coefficients $f[\ell]$
non-vanishing. Suppose that $f' \in \C^n$ satisfies 
$T^{D_n}_d(f')=T^{D_n}_d(f)$ for $d\leq 3$ then $f'\in \mathbb{R}^{n}$; i.e., $f'[-i]=\overline{f'[i]}$ for all $i$. 
\end{lemma}
\begin{proof}
The degree 2 invariants are the same as in the cyclic case and so we obtain ${f'}[i]{f'}[-i]\in \mathbb{R}$ implying that if $f'[i]=r'_ie^{i\varphi_i}$ then $f'[-i]=r'_{-i}e^{-i\varphi_i}$ $\forall i$. All that remains to show is that $r'_i=r'_{-i}$. 
As in the cyclic case, we have $f'[0]\in\mathbb{R}$ and when $n$ is even $f'[n/2]\in \mathbb{R}$ from the degree 1 and 2 invariants respectively.

The rest of the argument is similar to the cyclic case.
For each pair $(i,j)$ we know that the dihedral invariant $f[i]f[j]f[-i-j] + f[-i]f[-j]f[-i-j]$ is real so 
$$f'[i]f'[j]f'[-i-j]+f'[-i]f'[-j]f'[i+j]= r'_{i}r'_{j}r'_{-i-j}e^{\iota(\varphi_i+\varphi_j-\varphi_{i+j})}+r_{-i}r_{-j}r_{i+j}e^{\iota(-\varphi_i-\varphi_j+\varphi_{i+j})}$$ is also real.

Thus, 
$r'_{i}r'_{j}r'_{-i-j}\sin(\varphi_i+\varphi_j-\varphi_{i+j})-r'_{-i}r'_{-j}r'_{i+j}\sin(\varphi_i+\varphi_j-\varphi_{i+j})=0$ and hence $r'_ir'_jr'_{-i-j}=r'_{-i}r'_{-j}r'_{i+j}$ unless $\varphi_{i} + \varphi_j -\varphi_{i+j}$ is a multiple of $\pi$.
However, if $\varphi_{i} + \varphi_j -\varphi_{i+j}$ is multiple
of $\pi$ we can show that $r'_ir'_j r'_{-i-j} = r'_{-i}r'_{-j}r'_{i+j}$ with a different argument.
Since 
$f$ is the Fourier 
transform of a real vector, then the equality of second moments
$T_2(f') = T_2(f)$ implies that $r'_\ell r'_{-\ell} = r_\ell^2$ for; i.e.,
$r'_{-\ell} = r_\ell^2/r'_{\ell}$.  If $\sin(\varphi_i + \varphi_j -\varphi_{i+j}) = 0$ and we set
$\beta = r'_i r'_j r'_{-i-j}$ and $\alpha = r_i r_j r_{-i-j}$
then equality of degree-three invariants
$$f'[i]f'[j]f'[-i-j]+f'[-i]f'[-j]f'[i+j] = f[i]f[j]f[-i-j] + f[-i]f[-j]f[-i-j]  = 2\alpha \cos(\theta_i + \theta_j - \theta_{i+j})$$
implies that the non-negative real number $\beta$ satisfies the quadratic equation
\begin{equation} \label{eq.quadratic} \beta^2 -2 \beta\alpha \cos(\theta_i + \theta_j - \theta_{i+j}) + \alpha^2 =0
\end{equation}
which is impossible unless $\theta_i + \theta_j - \theta_{i+j}$ is also also multiple of $\pi$, since the discriminant of~\eqref{eq.quadratic} is $4\alpha^2(\cos^2(\theta_i + \theta_j -
\theta_{i+j}) - 1)$. In this case we must have $\beta = \alpha$
because they are both positive; i.e., $r'_ir'_jr'_{-i-j} = r_j r_j r_{i-j}$. We also know that $\varphi_{-i} + \varphi_{-j} -\varphi_{-i-j}$ is also a multiple of $\pi$ so the same argument implies that $r'_{-i}r'_{-j}r'_{i+j} = r_{-i}r_{-j}r_{i+j} = r_{i}r_{j}r_{-i-j}$, where the last equality holds because $f$ is real
\end{proof}
\begin{example}
  If some of the Fourier coefficients of a real vector $f$ vanish then there can be non-real vectors $f'$ with the same dihedral invariants of degree at most three. For example if $f \in \R^5$ is the vectors whose Fourier
  transform is $(1,1,0,0,1)$ then the complex (but not real) vector
$f' \in \C^5$ whose Fourier  transform is $(1,1/2,0,0,2)$ has the same $D_5$ invariants of degrees at most three. In order to separate these orbits we need to use invariants of degree five. 

In general, for a signal $x\in \mathbb{R}^n$ we require invariants of degree at least $n$ to separate all $D_n$ orbits \cite{kohls2010degree}.
\end{example}

\begin{example}
The real vector $f$ whose coordinates in the Fourier basis are
$$(f[0],f[1],f[2],f[3],f[4]) = (1,\iota, -\iota, \iota, -\iota)$$
has non-vanishing dihedral invariants of degree one and two.
However, the polynomial invariants of degree three 
$$f[1]f[2]^2 + f[4]f[3]^2, f[1]^2f[3] + f[4]^2f[2]$$
both vanish, and in this case the $D_n$ orbit is not uniquely determined by these invariants. Indeed the system of equations
\begin{equation}
\left\{\begin{array}{l}
f'[0]  =  1\\
f'[1]f'[4]  = 1\\
f'[2]f'[3]  =  1\\
f'[1]^2 f'[3] + f'[4]^2 f'[2] = 0\\
f'[1]f'[2]^2 + f'[4]f'[3]^2  =  0
\end{array}\right.
\end{equation}
has twenty solutions corresponding to two real orbits of the dihedral
group. Specifically the orbit of the real vector $f'$ whose Fourier coordinates
are $(1, -\iota, -\iota, \iota, \iota)$ has the same dihedral invariants but
is not dihedrally equivalent to $f$. Note that $f'$ does not
have the same cyclic invariants since $f'[1] f'[2]^2 = \iota$
but $f[1]f[2]^2 = -\iota$
\end{example}
\section{Open questions}
\subsection{Unitary invariants}
Tannaka-Krein duality implies that for the complex regular representation of a finite group (and more generally any compact group) the unitary invariants of degree at most three (the bispectrum) also separate generic orbits~\cite{smach2008generalized}.
By Theorem~\ref{thm.dihedral} we know that polynomial invariants of degree at
most three separate generic orbits in the `standard' representation $\C^n$ of $D_n$. However, the question as to whether the corresponding unitary invariants separate generic orbits remains open. For unitary invariants the term generic means that these invariants
separate orbits on a non-empty real Zariski open set in $\C^n = \R^{2n}$.
Since unitary and polynomial invariants agree on real vectors we know that unitary invariants separate generic vectors in the $\R^n \subset \C^n$ (ie vectors whose imaginary components are zero).
However, $\R^n$ is a proper closed subset of $\C^n$ in the real Zariski topology, so we cannot apply the density argument used in the proof of Proposition~\ref{prop.realtest}
to conclude that the generic orbit is separated by invariants of degree at most three.

\subsection{Invariants over other fields}
\cite[Theorem III.1]{edidin2025orbit} states that, for the regular representation defined over any infinite field, the polynomial invariants of degree at most three separate generic orbits. This leads to the question as to whether Theorem~\ref{thm.dihedral} holds for the standard $n$-dimensional representation defined over an arbitrary infinite field, or at least a field whose characteristic does not divide $2n$. The proof given here does not have an immediate extension to other fields. Likewise, the proof of \cite[Theorem III.1]{edidin2025orbit} makes essential use of the fact that the generic orbit in the regular representation consists of linearly independent vectors.

\section{Sample complexity and comparison with EM}

As is the case for the cyclic group, Theorem~\ref{thm.dihedral} establishes that the generic orbit can be identified from its third moment. However, the efficient frequency marching algorithm~\cite{bendory2017bispectrum} which works in the cyclic case becomes combinatorially difficult in the dihedral case. The reason is that reduction to a system of linear phase equations requires searching over $2^n$ possible conjugate configurations, most of which are inconsistent. Another approach that has been shown to work for the cyclic bispectrum is optimization.
In this approach we look for a global minimum of the the following loss function:
\begin{equation}
\label{eq:obj_func}
    \argmin_x \|\tilde{B} - B(x)\|_\text{F}^2,
\end{equation}
where~$\tilde{B}$ is the true bispectrum and~$B(x)$ is the computed bispectrum of a signal $x\in \mathbb{R}^n$. In~\cite{bendory2017bispectrum} the authors gave substantial evidence for the efficacy of this approach by showing empirically that this non-convex algorithm appears to converge to the
target signal with random initialization. Although this is a seemingly challenging polynomial optimization problem, when computed in the Fourier basis the entries of the bispectrum are degree 3 monomials and hence optimization is seemingly more feasible. In fact, because the bispectrum consists of monomials, the equations become linear under a logarithmic transformation.

We can attempt the same optimization procedure with the dihedral third moment $\tilde{T}$. However, because $D_n$ is no longer abelian, the entries of the third moment no longer consist of monomials. Instead, the entries are binomials of the form \begin{equation}
T(x)[k_1,k_2]=\hat{x}[k_1]\hat{x}[k_2]\hat{x}[k_3]+\overline{\hat{x}[k_1]\hat{x}[k_2]\hat{x}[k_3]}\end{equation} where $k_1+k_2+k_3\equiv 0 \mod n$. 

We will give numerical evidence that the corresponding dihedral optimization problem converges with random initialization. 

In the following experiments we sample from the model
\begin{equation} \label{eq:sampling}
y_i=g_ix+\epsilon_i
\end{equation}
where $g_i$ is taken from a uniform distribution on either $D_n$ or $\mathbb{Z}_n$. Following the methods in \cite{bendory2017bispectrum}, we construct an unbiased estimator for the bispectrum or dihedral third moment by mean centering the estimate. i.e we construct an estimate for the third moment tensor of $x-\mu(x)$.

\begin{equation*}
\tilde{M}_{x-\mu(x)}=\frac{1}{N}\sum_{i=1}^N M(y_i-\mu(x)) 
\end{equation*}
where $M=B$ is the bispectrum or $M=T$ is the reflection-invariant bispectrum.

After estimation we solve the resulting optimization problem with random initialization. We initialized the optimization algorithm from 10 different random initial
guesses and chose the one that yields the smallest value when evaluated in the cost function \eqref{eq:obj_func}. We will see that using more than one initialization is more valuable in the dihedral case where we occasionally see complete failure to converge to a local minimum whereas in the cyclic case this is not such an issue.

To exemplify this we looked at the errors incurred from 10 trials. \rev{In each trial we use a fixed vector $x\in \mathbb{R}^{37}$ whose entries are generated iid from $\mathcal{N}(0,1)$}. We take $10,000$ samples of the form~\eqref{eq:sampling} with $\sigma=2$ (we use the same vector for the cyclic and dihedral case). In the first regime each trial uses only one random initialization whereas in the second we take 10 random initializations and then use the one with lowest cost.

\begin{figure}[H]
    \centering
    \includegraphics[width=\textwidth, height=0.35\textheight, keepaspectratio]{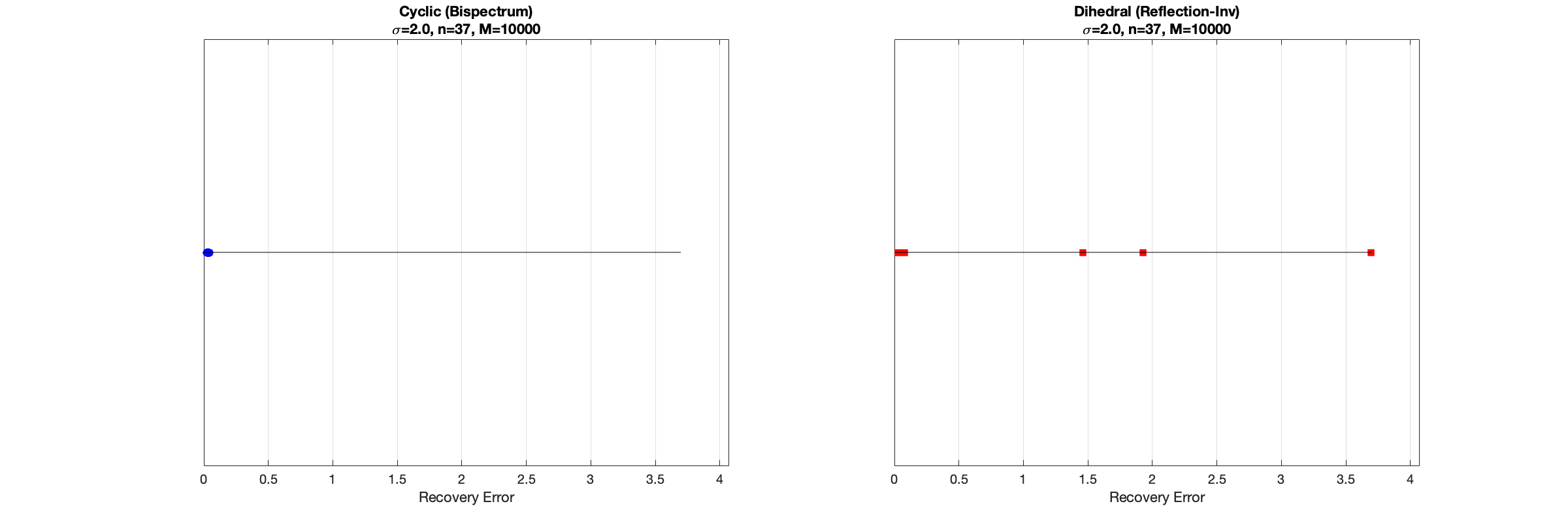}
    \label{fig:failure with 1 init}
\end{figure}

In the dihedral case, although 7 of the 10 trials approximately converge to the true vector, 3 of the trials completely fail to converge. In the cyclic case they all correctly converge with just one initialization.

Now we look at the improvement gains from using 10 initializations.

\begin{figure}[H]
    \centering
    \includegraphics[width=\textwidth, height=0.35\textheight, keepaspectratio]{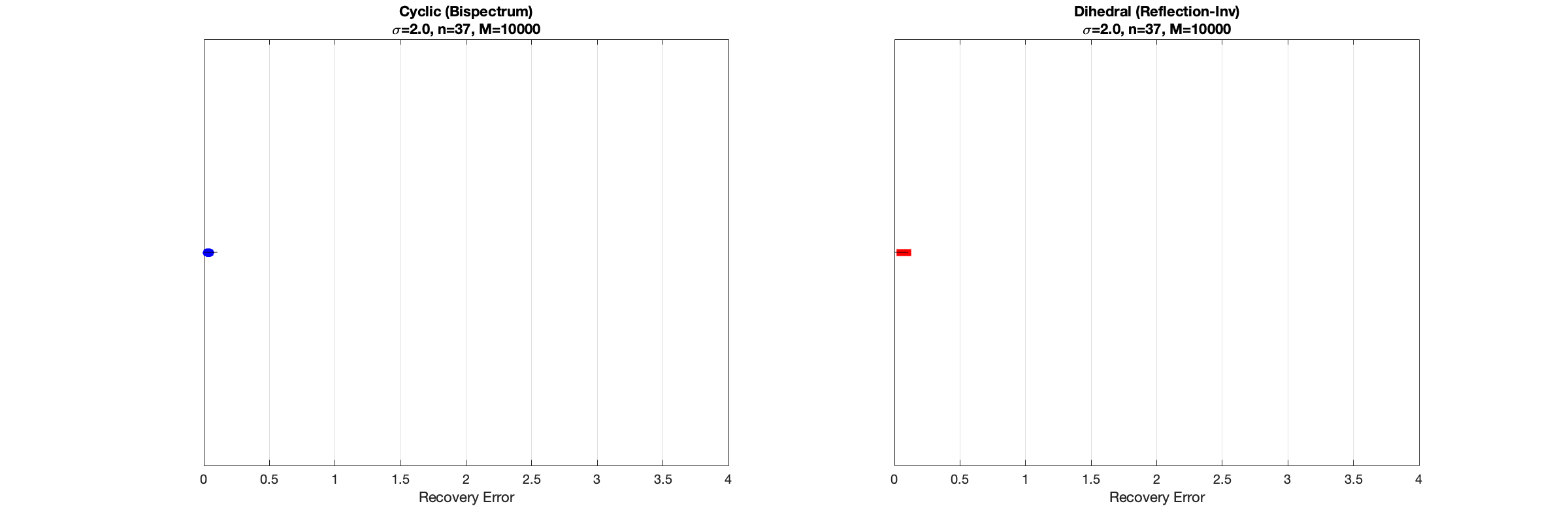}
    \label{fig:failure with 10 init}
\end{figure}

Hence going forward we will always use 10 random initializations and choose the one with the lowest cost (in the cyclic case this isn't necessary but we use the same regime for the sake of consistency).

In all trials we use a maximum of 200 iterations. We note that the cyclic case was implemented in \cite{bendory2017bispectrum} with optimization using the toolbox Manopt \cite{boumal2014manopt} which optimizes on the manifold of phases of the signal (as opposed to our direct implementation on the fourier coefficients). In keeping with the analysis in \cite{bendory2017bispectrum} for the cyclic group and \cite{bendory2022dihedral} for the dihedral group, we compare the moment based approach with the MLE approach implemented with expectation maximization. In keeping with \cite{bendory2022dihedral}, we initialized the EM algorithm from a single random point and halted it
when the difference of the likelihood between two consecutive iterations dropped below $10^{-4}$ or after a maximum of 400 iterations.

In the following experiments we use 10,000 samples with values of $\sigma$ between $.01$ and $10$. For each value of $\sigma$ we run 10 trials and then average the error of these trials.

\begin{figure}[H]
    \centering
    \includegraphics[width=\textwidth, height=0.35\textheight, keepaspectratio]{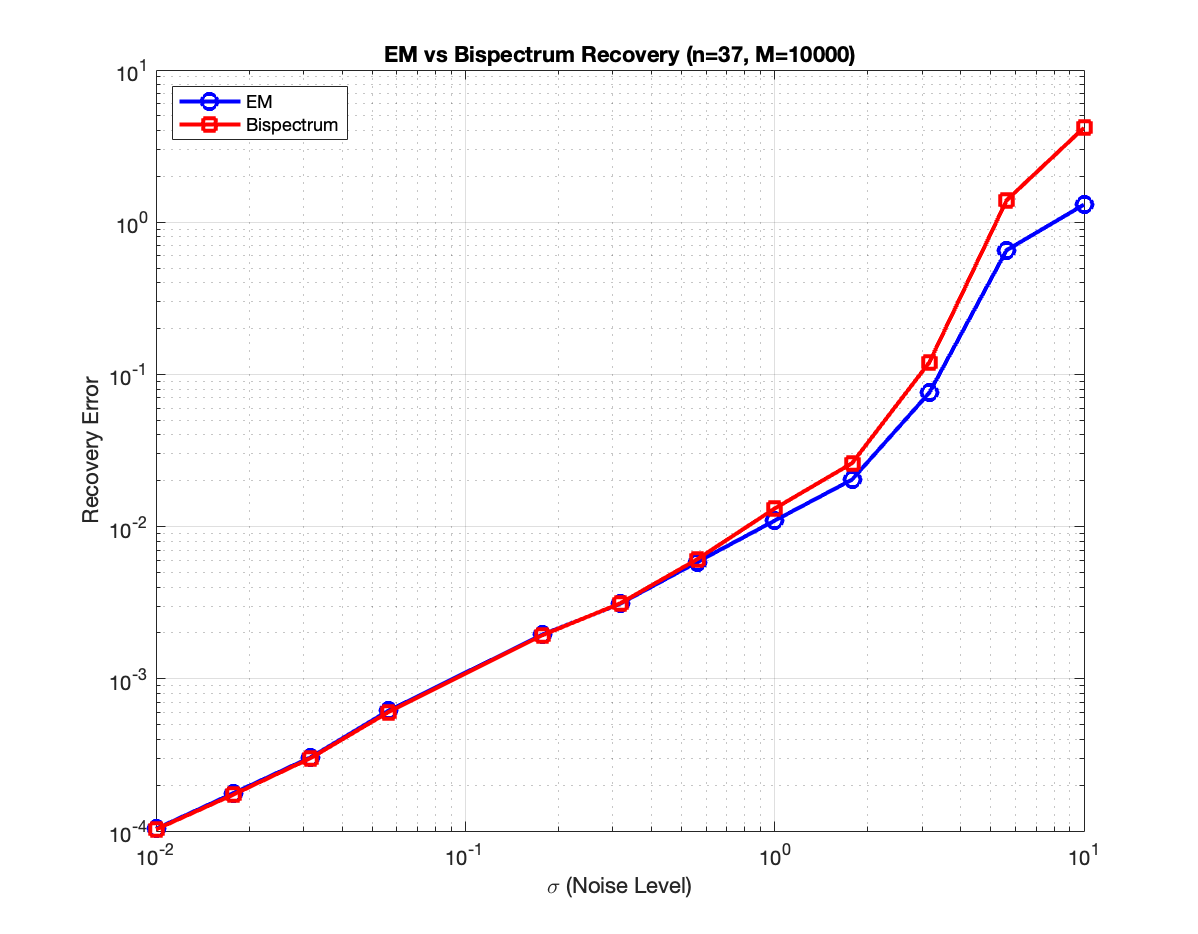}
    \label{fig:Cyclic bispectrum vs EM}
\end{figure}

\begin{figure}[H]
    \centering
    \includegraphics[width=\textwidth, height=0.35\textheight, keepaspectratio]{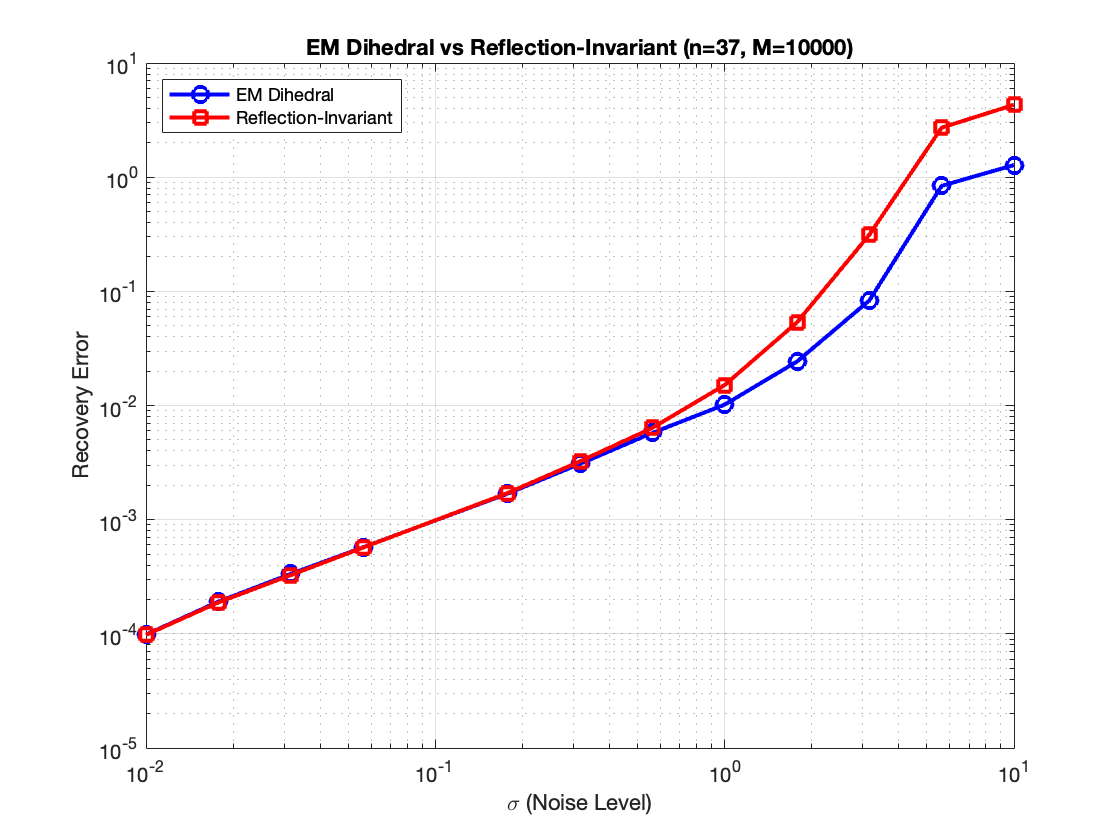}
    \label{fig:Dihedral invariant vs EM}
\end{figure}

When the signal length is below a certain threshold the error can be more substantial. This can be seen in the following plot which repeats the above experiment with a signal of length ranging from 5 to 100. This error is more substantial in the dihedral model than in the cyclic model. We keep $\sigma=1$ fixed.

\begin{figure}[H]
    \centering
    \includegraphics[width=\textwidth, height=0.35\textheight, keepaspectratio]{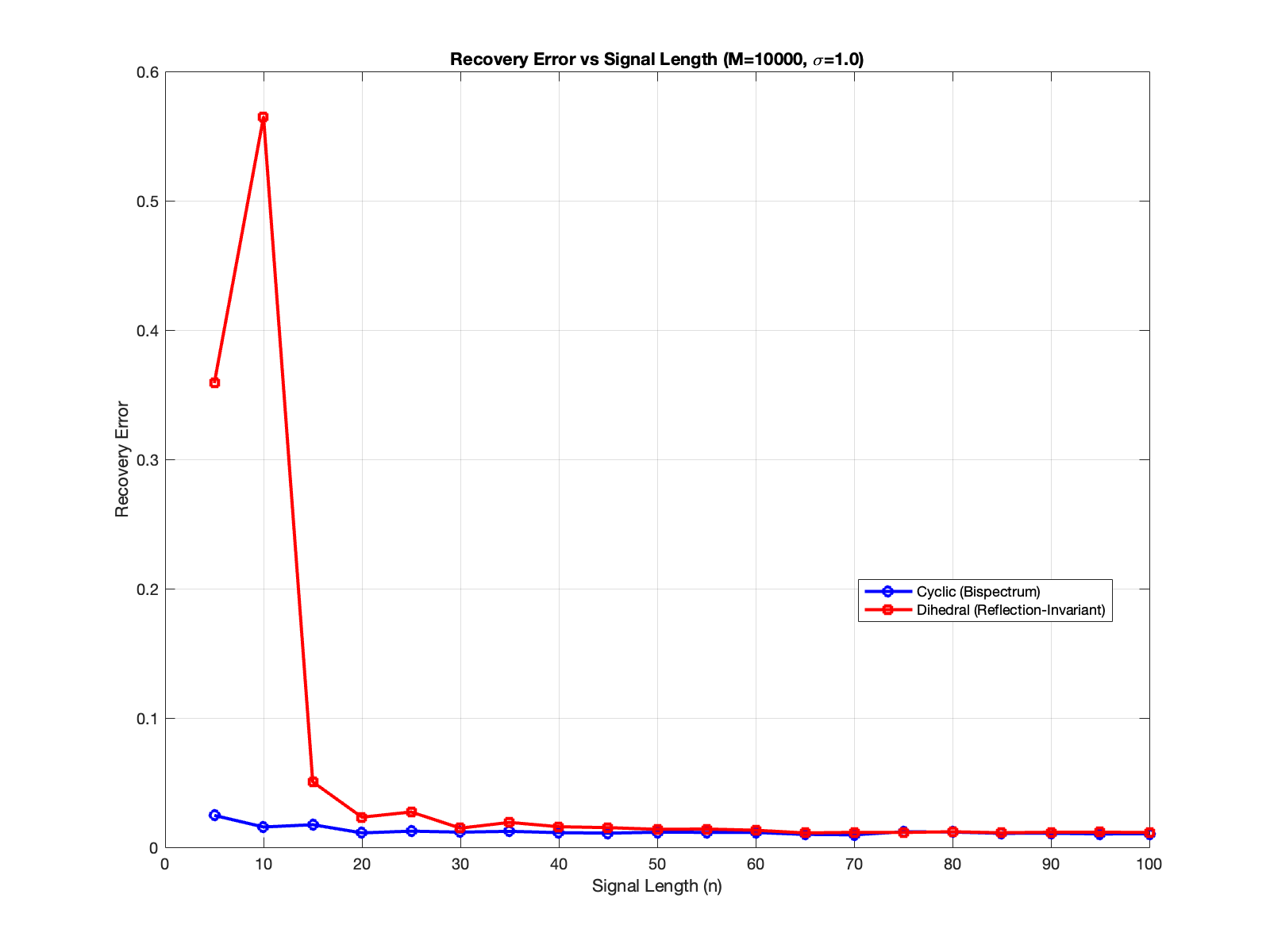}
    \label{fig:Dihedral invariant vs EM}
\end{figure}

The relatively high error for small length signals can be understood as a consequence of the growth rate of the number of terms in the bispectrum. The number of equations in the bispectrum is on the order of $\sim\frac{n^2}{6}$. Hence when $n>>0$, the number of equations quickly outgrows the dimension of the space of signals. Similarly, the number of equations in the dihedral third moment is on the order of $\sim\frac{n^2}{12}$ and hence when $n$ is small, the error will be more pronounced. As an example, when $n=5$ there are 9 distinct equations in the bispectrum which must solve for 5 unknowns. For the dihedral third moment there are exactly 5 equations in the 5 unknowns making the optimization much less robust.

\end{document}